\documentclass{amsart}

\usepackage{graphicx}
\usepackage{fancyhdr}
\usepackage{amsthm}
\usepackage{rotate}
\usepackage{graphicx}
\usepackage[T1]{fontenc}
\usepackage{afterpage}  %
\usepackage{comment}
\newtheorem{theorem}{Theorem}[section]
\newtheorem{lemma}[theorem]{Lemma}
\newtheorem{corollary}{Corollary}[theorem]
\theoremstyle{definition}
\newtheorem{definition}[theorem]{Definition}
\newtheorem{example}[theorem]{Example}

\theoremstyle{remark}

\DeclareMathOperator{\Perm}{Perm}

\numberwithin{equation}{section}

\DeclareUnicodeCharacter{2212}{-}
\begin{document}

\title{Operadic approach to HNN-extensions of Leibniz Algebras}


\author{Georg Klein}
\address{Departamento de Matem\'atica,
Universidade Federal da Bahia,
Av. Adhemar de Barros, S/N, Ondina, CEP: 40.170.110, Salvador, BA, BRAZIL}
\email{georgklein53@gmail.com}

\author{Chia Zargeh}
\address{Department of Mathematics and Computer Science, Modern College of Business and Science, Muscat, Sultanate of Oman}
\email{chia.zargeh@mcbs.edu.om}

\subjclass[2010]{05E15, 17A32, 17A99, 18D50}


\date{}

\keywords{operad, Leibniz algebras, HNN-extensions, Gr\"obner-Shirshov basis, di-algebra}

\begin{abstract}
We construct HNN-extensions of Lie di-algebras in the variety of di-algebras and
provide a presentation for the replicated HNN-extension of a Lie di-algebras. Then, by applying
the method of Gr\"obner-Shirshov bases for replicated algebras, we obtain a linear basis. As
an application of HNN-extensions, we prove that Lie di-algebras are embedded in their HNN-
extension. 
\end{abstract}

\maketitle



.

\section*{Introduction}
The Higman-Neuman-Neumann extension (HNN-extension) of a group was originally introduced in \cite{H1}, and has been used for the proof of a known embedding theorem, that every countable group is embeddable into a group with two generators. For a group $G$ with an isomorphism $\phi$ between two of its subgroups $A$ and $B$,  $H$ is an extension of $G$ with an element $t \in H$, such that $t^{-1}at= \phi(a)$ for every $a \in A$. The group $H$ is presented by  $$H=\langle G,t \mid t^{-1}at=\phi (a), a\in A \rangle$$
and it implies that $G$ is embedded in $H$. The  concept  of  HNN-extension  was  constructed  for  (restricted) Lie algebras in independent works by Lichtman and Shirvani [4] and Wasserman [19], and it has recently been extended to generalized versions of Lie algebras, namely, Leibniz  algebras,  Lie  superalgebras  and  Hom-setting  of  Lie  algebras in  \cite{L1}, \cite{L2} and \cite{S3}, respectively. As an application of HNN-extensions, Wasserman in \cite{W1} obtained some analogous results to group theory and proved that Markov properties of finitely presented Lie algebras are undecidable. Moreover, HNN-extensions are used to prove an essential theorem concerning embeddability of algebraic structures into two-generated ones.  
Ladra et al.~in \cite{L2} used the Composition-Diamond Lemma for di-algebras in order to prove that every dialgebra embeds into its HNN-extension, then by the PBW theorem they transferred their results to the case of Leibniz algebras. The Composition-Diamond Lemma (CD-Lemma, for short) is the key ingredient of Gr\"obner-Shirshov bases theory and it is a powerful tool in combinatorial algebra to solve various problems such as the normal form, the word problem, extensions and embedding theorems.  The theory of Gr\"obner-Shirshov bases is parallel to the theory of Gr\"obner bases and was introduced for ideals of free (commutative, anti-commutative) non associative algebras, free Lie algebras and simplicity free associative algebras by Shirshov (see \cite{B3}, \cite{S1} and \cite{S2}). It has been transferred  to different algebraic structures in the last twenty years. The reader interested in Gr\"obner-Shirshov bases and its applications is encouraged to study the recently published book by Bokut et al. \cite{B4} for a comprehensive account.
\\As there is no Gr\"obner-Shirshov basis for Leibniz algebras, Kolesnikov in \cite{G2} and \cite{K1} introduced a replication procedure based on operads in order to fill this gap. Accordingly, given the variety of Lie algebras governed by an operad $\mathcal{P}$, its replicated variety is denoted by di-Lie and it is governed by the Hadamard product of operads $\mathcal{P}{e}{r}{m}\otimes \mathcal{P}$ where $\mathcal{P}{e}{r}{m}$ is the operad corresponding to the variety of $\mathcal{P}{e}{r}{m}$-algebras. In this paper, we use Kolesnikov's approach in order to provide another version of HNN-extensions of Leibniz algebras in terms of operads. Leibniz algebras are non-antisymmetric generalization of Lie algebras introduced by Bloh \cite{B2} and Loday \cite{L4},\cite{L5}, and they have many applications in both pure and applied mathematics and in physics.  Because of this,  many known results of the theory of Lie algebras as well as combinatorial group theory have been extended to Leibniz algebras in the last two decades (see, for instance,  \cite{B1}, \cite{L1} and \cite{Z1}.)

This paper is organized as follows. In the first section of this paper, we recall the notions of operad and replicated algebras. In the second section, the theory of Gr\"obner-Shirshov bases for Lie algebras will be recalled. Then we implement the replication procedure to define Gr\"obner-Shirshov bases for HNN-extensions of Lie di-algebras (Leibniz algebras) in the third section. 
\section{General overview on operads and replicated algebras}
Operads were first introduced in algebraic topology. Recently, the theory of operads has experienced further developments in several directions and has been used to investigate complicated algebraic structures. For our purpose, operads based on the category of vector spaces play an important role in the replication procedure to determine Gr\"obner-Shirshov bases. 
In this section, we recall fundamental concepts of operads from the algebraic point of view. For a detailed explanation  of the notion of operads, the reader is referred to \cite{L7}. There are several definitions of operads. For instance, the monoidal definition of operads and the combinatorial definition. The first one is based on the concept of  monoidal category of $S$-modules and the latter is given by rooted trees for symmetric operads and by planar rooted trees for nonsymmetric operads. We provide the monoidal definition of operads in what follows.

Let $\mathcal{P}$ be a variety of algebras defined by polynomial identities (for instance, associative algebra, Lie algebra, etc...) and  denote by $\mathcal{P}(V)$ the free algebra over the finite-dimensional vector space $V$. $\mathcal{P}$ can be considered a functor from the category of finite-dimensional vector spaces  $Vect$ to itself. In addition, the map $V \to \mathcal{P}(V)$ gives a natural transformation $Id \to \mathcal{P}$. By Schur's lemma,  $\mathcal{P}(V)$ is of the form 
\[ \mathcal{P}(V)= \oplus_{n \geq 1} \mathcal{P}(n) \otimes_{S_{n}} V^{\otimes n},\]
for some right $S_{n}$-module $\mathcal{P}(n)$, where $S_{n}$ is the permutation group. By applying the universal property of the free algebra to $Id: \mathcal{P}(V) \to \mathcal{P}(V)$, one obtains a natural map $\mathcal{P}(\mathcal{P}(V)) \to \mathcal{P}(V)$. This map is a transformation of functors $\gamma: \mathcal{P} \circ \mathcal{P} \to \mathcal{P}$, which is associative with unit. 
\begin{definition}
An algebraic operad is a functor $\mathcal{P}: Vect \to Vect,$ such that $\mathcal{P}(0)=0$, equipped with a natural transformation of functors $ \mathcal{P} \circ \mathcal{P} \to  \mathcal{P}$ which is associative with unit $1: Id \to \mathcal{P}$.
\end{definition}
For two operads $\mathcal{P}_{1}$ and $\mathcal{P}_{2}$, a morphism from  $\mathcal{P}_{1}$ to $\mathcal{P}_{2}$ is a family $f=\{f_{n}\}_{n \geq 1}$ of linear maps $ f_{n}:\mathcal{P}_{1}(n) \to \mathcal{P}_{2}(n)$ for $n \geq 1$ preserving the composition rule and the identity. 

Let us consider $\mathcal{P}(V)$ and $\mathcal{P}\circ \mathcal{P}(V)$ in terms of vector spaces of the $n$-ary operations defined on a type of algebra denoted by $\mathcal{P}(n): V^{\otimes n} \to V $. Then the operation $\gamma$ is defined by the linear maps
\[ \mathcal{P}(n) \otimes \mathcal{P}(i_{1}) \otimes \cdots \otimes \mathcal{P}(i_n) \to \mathcal{P}(i_{1}+ \cdots + i_{n}), \]
and an algebra of type $\mathcal{P}$ is defined by $\mathcal{P} \otimes_{S_{n}} V^{\otimes n} \to V$. The family of $S_{n}$-modules $\{ \mathcal{P}(n)\}_{n \geq 1}$ is called an $S$-module. The left adjoint functor exists and gives rise to the free operad over an $S$-module.

Quadratic operads govern varieties of binary algebras defined by multilinear identities of degree $2$ and $3$. In what follows, we intend to describe quadratic operads $\mathcal{L}{i}{e}$ and $\mathcal{P}{e}{r}{m}$, which are governing the variety of Lie algebras and $\Perm$-algebras (see~\cite{Z2} for $\Perm$-algebras), respectively.  Indeed, by considering a space of binary operations denoted by $E$ and the space of relations $R$, a quadratic operad can be defined. 
\begin{definition}\cite{G3}
A binary operad is an operad generated by operations on two variables. More explicitly, Let $E$ be an $S_{2}$-module (module of generating operators) and let $\mathcal{F}(E)$ be the free operad on the $S$-module $(0,E,0,\dots )$ where the action of $S_{2}$ on $E \otimes E$ is on the second factor only. In other words, the action of $(12) \in S_{2}$ on $E \otimes E$ is given by $id \otimes (12)$. A binary operad is quadratic if it is the quotient of $\mathcal{F}(E)$ by the ideal generated by some $S_{3}$-submodule $R$ of $\mathcal{F}(E)(3)$ 
\[ \mathcal{F}E(3)=\mathbb{K}S_{3} \otimes_{S_{2}} ( E \otimes E), \] 
being $\mathcal{F}(E)(3)$ is the space of all the operations on $3$ variables which can be performed out of the operations of $2$ variables.
\end{definition}

\begin{example}\label{Lieoperad}\cite{G3}
The space $E$ of binary operations (Lie bracket) is considered as $S_{2}$-module. If $\mu$ is an element of $E$ representing a  Lie bracket as $(x_{1},x_{2}) \mapsto [x_{1},x_{2}]$ then $\mu^{(12)}$ corresponds to $(x_{1},x_{2}) \mapsto [x_{2},x_{1}]$ and $\mu^{(21)} = - \mu$. The space of multilinear term of degree $3$ is identified with 
$ \mathcal{F}E(3)=\mathbb{K}S_{3} \otimes_{S_{2}} ( E \otimes E)$ defined as above. We have $\mu \otimes \mu \in E \otimes E$ and the elements of $E(3)$ are described as:
\[
\begin{tabular}{|l|c|}
\hline
$\mathbb{K}S_{3} \otimes_{S_{2}} ( E \otimes E)$ & Appropriate Monomial \\
\hline
$1 \otimes_{S_{{2}}}
(\mu \otimes \mu)$ & $[[x_{1},x_{2}],x_{3}]$ \\
$1 \otimes_{S_{{2}}}
(\mu \otimes \mu^{(12)})$ & $[[x_{2},x_{1}],x_{3}]$ \\ 
$1 \otimes_{S_{{2}}}
(\mu^{(12)} \otimes \mu)$ & $[x_{3},[x_{1},x_{2}]]$ \\
$1 \otimes_{S_{{2}}}
(\mu^{(12)} \otimes \mu^{(12)})$ & $[x_{3},[x_{2},x_{1}]]$ \\ 
\hline
$(13) \otimes_{S_{{2}}}
(\mu \otimes \mu)$  & $[[x_{3},x_{2}],x_{1}]$ \\
$(13) \otimes_{S_{{2}}}
(\mu \otimes \mu^{(12)})$ & $[[x_{2},x_{3}],x_{1}]$\\
$(13) \otimes_{S_{{2}}}
(\mu^{(12)} \otimes \mu)$  & $[x_{1},[x_{3},x_{2}]]$\\
$(13) \otimes_{S_{{2}}}
(\mu^{(12)} \otimes \mu^{(12)})$ & $[x_{1},[x_{2},x_{3}]]$ \\

\hline
$(23) \otimes_{S_{{2}}}
(\mu \otimes \mu)$ & $[[x_{1},x_{3}],x_{2}]$ \\
$(23) \otimes_{S_{{2}}}
(\mu \otimes \mu^{(12)})$ & $[[x_{3},x_{1}],x_{2}]$  \\ 
$(23) \otimes_{S_{{2}}}
(\mu^{(12)} \otimes \mu)$ &$[x_{2},[x_{1},x_{3}]]$\\
$(23) \otimes_{S_{{2}}}
(\mu^{(12)} \otimes \mu^{(12)})$ &  $[x_{2},[x_{3},x_{1}]]$\\ 

\hline
\end{tabular}
\]
The operad $\mathcal{L}{i}{e}$ is the quotient $\mathcal{L}{i}{e}= E_{Lie} / R_{Lie}$, where $E_{Lie}$ is the space of Lie bracket and $R_{Lie}$ is the $S_{3}$-submodule of $E\otimes E$ generated by the Jacobiator which is the following sum 
\[1 \otimes_{S_{{2}}} (\mu \otimes \mu) + (13) \otimes_{S_{{2}}}
(\mu \otimes \mu^{(12)}) + (23) \otimes_{S_{{2}}}
(\mu \otimes \mu^{(12)}) =0.\]
\end{example}
In the next example, the operad governing the variety of associative algebras satisfying the identity $(x_{1}x_{2})x_{3}=(x_{2}x_{1})x_{3}$ is introduced. It is called the $\mathcal{P}{e}{r}{m}$ operad. For more details see \cite{G2}. 
\begin{example}\label{permoperad}
Let $\mathcal{P}{e}{r}{m}$ be the variety of associative algebras satisfying the  identity \[(x_{1}x_{2})x_{3}=(x_{2}x_{1})x_{3}.\]
Let $e_{i}^{(n)}=(x_{1} \cdots x_{i-1}x_{i+1}\cdots x_{n})x_{i},$ $i=1,\ldots,n$. Then $\{e_{1}^{(n)}, \ldots , e_{n}^{(n)} \}$ is a standard basis. Define the composition  and the action of the symmetric group $S_{n}$ as follows:
\[ \gamma_{m_{1},\ldots,m_{n}} (e_{i}^{(n)} \otimes e_{j_{1}}^{(m_{1})} \otimes \cdots \otimes e_{j_{n}}^{(m_{n})}) = e_{m_{1}+ \cdots + m_{n-1}+j_{i}}^{(m)}\]
and 
\[ \sigma : e_{i}^{(n)} \mapsto e_{i \sigma}^{(n)} \]
where $m_{1}+\dots+m_{n} = m$. This yields a symmetric operad denoted by $\mathcal{P}{e}{r}{m}$. 
\end{example}
Given a variety which is governed by the operad $\mathcal{V}{a}{r}$, its replicated variety is denoted by di-$\mathcal{V}{a}{r}$ and governed by the following Hadamard product of operads
\[ \text{di}\text{-} \mathcal{V}{a}{r}=\mathcal{V}{a}{r} \otimes \mathcal{P}{e}{r}{m} .\]
The Hadamard tensor product has a natural operad structure and the composition maps are expanded componentwise. A permutation $\sigma \in S_{n}$ acts on the Hadamard product as $\sigma \otimes \sigma$.  For $A \in \mathcal{V}{a}{r}$, $P \in \mathcal{P}{e}{r}{m}$, the tensor product  $A \otimes P$ equipped with the operations 
\begin{equation}\label{eq:1}
f_{i}(x_{1} \otimes a_{1}, \ldots , x_{n} \otimes a_{n})= f(a_{1},\ldots , a_{n}) \otimes e_{i}^{(n)}(x_{1} , \ldots , x_{n}) ,
\end{equation} 
\[ f \in \Sigma \text{, }  \nu(f)=n \text{, }  x_{i} \in P  \text{, }  a_{i} \in A  \text{, }  i=1,\ldots,n\]
belongs to the variety di-$\mathcal{V}{a}{r}$. In general, finding the generators and defining relations of the Hadamard product is difficult. However, if $\mathcal{V}{a}{r}$ is a binary quadratic operad, i.e. an operad corresponding to a variety whose defining identities have degrees $2$ or $3$, the Hadamard tensor product $\mathcal{V}{a}{r} \otimes \mathcal{P}{e}{r}{m}$ coincides with the Manin white product and is denoted by $\mathcal{V}{a}{r} \circ \mathcal{P}{e}{r}{m},$ \cite{G3}. Accordingly, it is proved that $\mathcal{L}{e}{i}{b}=\mathcal{L}{i}{e} \circ \mathcal{P}{e}{r}{m}$, where $\mathcal{L}{e}{i}{b}$ is the operad governing the variety of Leibniz algebras. 
\\In order to clarify the operations  (\ref{eq:1}), we construct $[x,y] \otimes e_{1}^{(2)}$, where $e_{1}^{(2)} \in \mathcal{P}{e}{r}{m}(2)$. We have 
\begin{align*}
    [x,y]\otimes e_{1}^{(2)}&= xy \otimes e_{1}^{(2)} - yx \otimes e_{1}^{(2)}\\
                            &= xy \otimes e_{1}^{(2)} + ((xy) \otimes e_{2}^{(2)})^{(12)}
\end{align*}
then the Leibniz multiplication is defined as 
   \[ [x \dashv y] = x \dashv y - y \vdash x,\]
which implies the right Leibniz algebra. By computing $[x,y] \otimes e_{1}^{(2)}$, the left Leibniz algebra will be obtained in a similar way.

Kolesnikov \cite{K1} provided a construction of the free tri-$\mathcal{V}{a}{r} \langle X \rangle$ generated by a given set $X$ in the variety tri-$\mathcal{V}{a}{r}$ and obtained essential results. In what follows, we obtain analogous results in the  di-$\mathcal{V}{a}{r}$ setting.

\subsection*{Replication of variety of Lie algebras \cite{G2,K1}}
Kolesnikov~\cite{K1} provided a construction of the free tri-$\mathcal{V}{a}{r} \langle X \rangle$ generated by a given set $X$ in the variety tri-$\mathcal{V}{a}{r}$ and obtained essential results. In what follows, analogous results in the  di-$\mathcal{V}{a}{r}$ setting will be provided in accordance with~\cite{G2} and~\cite{K1}.
Let $(\Sigma , \nu)$ be a set of operations together with an arity function $\nu: \Sigma \to \mathbb{Z}_{+}$ which is called a \textit{language}. A $\Sigma$-algebra is a linear space $A$ equipped with polylinear operations $f: A^{\otimes n} \to A$, $f \in \Sigma$, $n=\nu(f)$. Denote by Alg $=$ {Alg$_\Sigma$} the class of all $\Sigma$-algebras, and let  {Alg$_\Sigma$}$\langle X \rangle$ be the free algebra generated by $X$.  A replicated language is defined as follows:
\[\Sigma^{(2)} = \{f_{i} \mid f \in \Sigma, i=1,\dots, \nu(f) \},~ \nu^{(2)}(f_{i})=\nu(f).\]
Let us assume that $\Sigma$ consists of one binary operation, then {Alg$_\Sigma$}$\langle X \rangle$  is a magmatic algebra and the replicated language $\Sigma^{(2)}$  is the set of binary operations $\{ \vdash_{i} , \dashv_{i}  \mid i \in I \}$.  We put  Alg$^{(2)}$ = Alg{$_\Sigma{^{(2)}}$}.
\\Let $\mathcal{V}{a}{r}$ be a variety of $\Sigma$-algebras satisfying a family of polylinear identities $S(\mathcal{V}{a}{r}) \subset ${Alg$_\Sigma$}$\langle X \rangle$.  We obtain a $\Sigma^{2}$-identity in  $S(\mathcal{V}{a}{r})$ by replacing all products with either $\dashv$ or $\vdash$ in such a way that all horizontal dashes point to a selected variable. 
To illustrate this, we have the following example. 

\begin{example}
Let $\lvert \Sigma \rvert=1$ and  $\mathcal{L}{i}{e}$ be the variety of Lie algebras. Then $\Sigma^{(2)}$ includes two binary operations $f_{1}=[x \dashv y]$ and $f_{2}=[x \vdash y]$.

\end{example}
In what follows,  we recall the concept of averaging operators which provides an equivalent definition for di-$\mathcal{V}{a}{r}$. 
\begin{definition}\cite{G2}
Suppose $A$ is a $\Sigma$-algebra. A linear map $t : A \to A $ is called an averaging operator on $A$ if 
\[ f(ta_{1}, \dots , ta_{n})=tf(ta_{1},\dots,ta_{i-1},a_{i},ta_{i+1},\dots,ta_{n})\]
for all $f \in \Sigma$, $ \nu(f)=n,$ $a_{j} \in A$, $ i,j=1,\dots,n$. The operator $t$ is called a homomorphic averaging operator if $f(ta_1,\dots,ta_n)=tf(a_{1}^H,\dots,a_{n}^H)$, where $H$ is in the collection of all nonempty subsets of $\{1,\dots,n\}$, and

\begin{align}
a_{i}^{H}=\left\{\begin{array}{lll}
                a_{i}  & i \in H \\
                ta_{i} & i \notin H.
            \end{array}\right.
\end{align}

\end{definition}
If $t$ is an averaging operator of the  $\Sigma$-algebra $A$, denote by $A^{(t)}$ the following $\Sigma^{(2)}$-algebra 
\[f^{H}(a_1,\dots,a_{n})=f(a_{1}^{H},\dots,a_{n}^{H}),\]
where $f \in \Sigma$, $\nu(f)=n$, $a_{i} \in A$ and $a_{i}^{H}$ are given by (2). Then the next theorem provides an equivalent definition of di-$\mathcal{V}{a}{r}$ by means of averaging operators.
\begin{theorem}\cite{G2}
Suppose $\nu(f) \geq 2$ for all $f \in \Sigma$.
\begin{itemize}
    \item[1.] If $A \in \mathcal{V}{a}{r}$ and $t$ is an averaging operator on $A$ then $A^{(t)}$ is a {\em di-}$\mathcal{V}{a}{r}$ algebra.
    \item[2.] Every $D \in $ \text{\em di-}$\mathcal{V}{a}{r}$ may be embedded into $A^{(t)}$ for an appropriate $A \in \mathcal{V}{a}{r}$ with an averaging operator $t$.
\end{itemize}
\end{theorem}
Let us consider the free Lie algebra $ Lie \langle X \cup \dot{X} \rangle$ generated by a given set $X$ and its copy $\dot{X}=\{\dot{X} \mid x \in X \}$  in the variety $\mathcal{L}{i}{e}$. There exists a unique homomorphism $\phi : Lie\langle X\cup \dot{X} \rangle \to Lie \langle X \rangle $ defined by $ x \mapsto x$, $ \dot{x} \mapsto x$, $ x \in X$. We define $Lie^{(2)} \langle X \cup \dot{X} \rangle$ with the following binary operations:
\[f \vdash g=\phi(f)g,~ f \dashv g=f\phi(g) \]
for $f,g \in Lie \langle X\cup \dot{X} \rangle$.

\begin{lemma}
The algebra  $Lie^{(2)} \langle X \cup \dot{X} \rangle$ belongs to di-$\mathcal{L}{i}{e}$.
\end{lemma}
\begin{proof}
We have $\phi^{2}=\phi$, so
\[ \phi(\phi(f)g)=\phi(f\phi(g))=\phi(fg) \text{ for }  f,g \in Lie^{(2)} \langle X \cup \dot{X} \rangle,\]
which shows that $\phi$ is an averaging operator on $Lie^{(2)} \langle X \cup \dot{X} \rangle$. Thus $Lie^{(2)} \langle X \cup \dot{X} \rangle$ belongs to  di-$\mathcal{L}{i}{e}$.
\end{proof}
\begin{lemma}
The subalgebra $V$ of $Lie^{(2)} \langle X \cup \dot{X} \rangle$ generated by $\dot{X}$ coincides with the subspace $W$ of $Lie \langle X \cup \dot{X} \rangle$ spanned by all monomials $u$ such that the degree of $u$ with respect to the variables from $\dot{X}$ is not zero, i.e.  $deg_{\dot{X}} u >0$.
\end{lemma}
\begin{proof}
Let $u,v \in V$ and assume $u,v \in W$. Then $u \vdash v$ and $u \dashv v$ belong to $W$ and $V \subset W$. The converse $W \subset V$ is proved in the same way.
\end{proof}
The next theorem has been proved in the case of tri-algebras in \cite{K1}. We restate the theorem in terms of free Leibniz algebras in the variety $\text{di-}\mathcal{L}{i}{e}$.
\begin{theorem}\label{theoremsubspace}
The subalgebra $V$ of $Lie^{(2)} \langle X \cup \dot{X} \rangle$ is isomorphic to the free Leibniz algebra in the variety di-$\mathcal{L}{i}{e}$ generated by $X$. 
\end{theorem}
\begin{proof}
We prove the universal property of $V$ in the class di-$\mathcal{L}{i}{e}$. Suppose $A$ is an arbitrary Leibniz algebra in di-$\mathcal{L}{i}{e}$, and let $\alpha: X \to A$ be an arbitrary map. We construct a homomorphism of Leibniz algebras $\chi :V \to A$ such that $\chi(\dot{x})=\alpha(x)$ for all $x \in X$. The subspace 
\[ A_{0}=\text{Span}\{a \vdash b - a \dashv b \mid a,b \in A \} \]
is an ideal of $A$. The quotient $\bar{A}=A/A_{0}$ has the structure of a Lie algebra. Consider $\hat{A}=\bar{A}\oplus A$ with the product:
\[ \bar{a}b=a\vdash b, ~ a\bar{b}=a \dashv b, \]
for $\bar{a},\bar{b} \in \bar{A}$, where $\bar{c}=c+A_{0} \in \bar{A},$ $c \in A$. Then $\hat{A} \in \mathcal{L}{i}{e}$. We recall that the Hadamard tensor product $\mathcal{P}{e}{r}{m} \otimes \mathcal{L}{i}{e}$ leads to Leibniz algebras in the variety di-$\mathcal{L}{i}{e}$. 
We construct $\hat{\alpha}: X \cup \dot{X} \to \hat{A}$ as ${\hat{\alpha}}(x)=\overline{\alpha x} \in \bar{A},$ $\hat{\alpha}( \dot{x})=\alpha(x) \in A$ for $x \in X$. The map $\hat{\alpha}$ induces a homomorphism of Lie algebras $\hat{\psi}: Lie \langle X \cup \dot{X} \rangle \to \hat{A}$. 

\end{proof}

\begin{example}
Suppose that $X$ is a linearly ordered set and extend the order to $X \cup \dot{X}$ as $$x > y \Rightarrow \dot{x} > \dot{y}, ~ \dot{x}> y$$ for all $x,y \in X$. All words of the form $[\cdots[[\dot{x_{i_{1}}}x_{i_{2}}]x_{i_{3}}]\cdots x_{i_{n}}]$ are linearly independent in $Lie \langle X \cup \dot{X} \rangle$. These words correspond to 
\[ [\cdots [[x_{i_{1}} \dashv x_{i_{2}}]\dashv x_{i_{3}}]\cdots \dashv x_{i_{n}}] \in \text{di-Lie}(X) ,\]
where $[~ \dashv ~]$ satisfies the right Leibniz identity. Therefore, the words of this form are a linear basis of di-Lie$(X)$.
 
\end{example}

\section{CD-Lemma for Lie algebras}
In this section, we briefly recall the main concepts of Gr\"obner-Shirshov bases theory for Lie algebras referring to \cite{B5}. Let $X$ be an ordered set and denote by $X^{*}$ and $X^{**}$ the set of all associative and non-associative words on $X$, respectively. We also consider a monomial ordering on $X^{*}$. A standard example of monomial ordering on $X^{*}$ is the \textit{deg-lex} ordering. An associative Lyndon-Shirshov word is a word greater than any cyclic permutation of itself. A non-associative Lyndon-Shirshov word is obtained by a unique standard bracketing and defined as follows. 
\begin{definition}
Let $[u]$ be a non-associative word. Then $[u]$ is called a non-associative Lyndon-Shirshov word if
\begin{enumerate}
    \item $u$ is an associative Lyndon-Shirshov word.
    \item if $[u]=[[v][w]]$, then both $[v]$ and $[w]$ are non-associative Lyndon-Shirshov words.
    \item in (2) if $[v]=[[v_{1}][v_{2}]]$, then $v_{2} \leq w$ in $X^{*}$.
\end{enumerate}
\end{definition}
Given a Lie polynomial $f \in Lie \langle X \rangle$, we can express it as a linear combination of non-associative Lyndon-Shirshov words. It is obvious that the leading term, $\bar{f}$ is an associative Lyndon-Shirshov word. 

\begin{definition}
For two monic Lie polynomials $f$ and $g$ in  $Lie \langle X \rangle$, their composition is denoted by $(f,g)_{w}$ and defined as follows:
\begin{itemize}
    \item Let $w=\bar{f}a=b \bar{g}$, then $(f,g)_{w}=[fb]_{\bar{f}} - [ag]_{\bar{g}}$  such that $deg(\bar{f})+ deg(\bar{g}) > deg(w)$ is called the intersection composition.
    \item Let $w=\bar{f}=a\bar{g}b$, then the composition $(f,g)_{w}=f - [agb]_{\bar{g}}$ is called the inclusion composition. 
\end{itemize}
where $a,b \in X^{*}$. The notations $[fb]_{\bar{f}}$, $[ag]_{\bar{g}}$ and $[agb]_{\bar{g}}$ imply a special bracketing with respect to the subword $\bar{f}$ or ${\bar{g}}$. 
\end{definition}
Let $S \subset Lie \langle X \rangle$ be a set of monic Lie polynomials, the composition $(f,g)_{w}$  is called trivial modulo $S$ if 
\[(f,g)_{w}=\sum \alpha_{i}a_{i}s_{i}b_{i} \]
where $\alpha_{i} \in k$, $a_{i},b_{i} \in X^{*}$ and $s_{i} \in S$ with $a_{i}\bar{s_{i}}b_{i} < w$. 
\begin{definition}
The set $S \subset Lie \langle X \rangle$ is called Gr\"obner-Shirshov basis if any composition of polynomials from $S$ is trivial modulo $S$.
\end{definition}
The following lemma is called \textit{Composition-Diamond Lemma} for  Lie algebras.
\begin{lemma}(CD-Lemma for Lie algebras, \cite{B5})\label{cdlemma}
Let $S \subset Lie (X) \subset k \langle X \rangle$ be nonempty set of monic Lie polynomials. Let $Id(S)$ be the ideal of $Lie(X)$ generated by $S$. Then the following statements are equivalent.
\begin{enumerate}
    \item  $S$ is a Gr\"obner-Shirshov basis in $Lie(X).$
    \item  $f \in Id(S) \Rightarrow \bar{f}=a\bar{s}b,~ \text{for some}~ s\in S~ \text{and}~ a,b \in X^{*}.$
    \item $Irr(S)=\{[u] \mid [u]~ \text{is a non-associative LS-word},~ u \neq a\bar{s}b, s \in S, a,b \in X^{*} \}$ is a basis for $Lie (X|S)$.
\end{enumerate}
\end{lemma}
In order to determine a Gr\"obner-Shirshov basis of an ideal of a Leibniz algebra $L \in \text{di-}\mathcal{L}{i}{e}$,  generated by $X$ with the set of relations $S$, we rewrite the relations from $S$ as elements of $Lie \langle X \cup \dot{X} \rangle$ and find a Gr\"obner-Shirshov basis of  the ideal $I=(S \cup \phi(S))$  in $Lie \langle X \cup \dot{X} \rangle$. The following theorem in \cite{K1} is useful for the translation of a Gr\"obner-Shirshov basis from the variety $\mathcal{L}{i}{e}$ to $\mathcal{L}{e}{i}{b}$. 

\begin{theorem}
Let $S \subset V \subset Lie^{(2)} \langle X \cup \dot{X} \rangle$. Then $(S)^{(2)}=(S \cup \phi(S)) \cap V$, where $(S)$ stands for the ideal of $Lie^{(2)} \langle X \cup \dot{X} \rangle$ generated by $S$.
\end{theorem}
\begin{proof}
Let  $ I= (S \cup \phi(S))$. We have 
\[ I=\cup_{s \geq 0} I_{s}, ~ I_{0} \subset I_{1} \subset \cdots ,\]
where $I_{0}=span(S \cup \phi(S)),$ and \[I_{s+1}=I_{s}+Lie \langle X \cup \dot{X} \rangle I_{s} + I_{s} Lie \langle X \cup \dot{X} \rangle.\]
We also consider $J=(S)^{(2)}$. Then
\[J=\cup_{s \geq 0} J_{s},~ J_{0} \subset J_{1} \subset \cdots ,\]
where $J_{0}=span(S),$ and
\[J_{s+1}=J_{s}+ V \vdash J_{s} + J_{s}\dashv V + V \dashv J_{s} + J_{s} \vdash V+VJ_{s}+J_{s}V \text{.}\]
Since $ S \subset V$, $\phi(S) \subset Lie (X)$, and $Lie(X) \cap V =0$, we have $I_{0} \cap V=J_{0}$. Moreover, $I_{0}=J_{0}+I'_{0},$ where $I'_{0}=\text{span}\phi(S)=\phi(J_{0})\subset Lie(X)$. Suppose $I_{s}=J_{s}+\phi(J_{s})$ for some $s \geq 0$. We have $Lie \langle X \cup \dot{X} \rangle=V+Lie(X)$ and $Lie(X)=\phi(V)$. Then
\begin{align*}
    I_{s+1}&=I_{s}+(V+Lie(X))I_{s}+I_{s}(V+Lie(X))\\
           &=J_{s}+\phi(J_{s})+VJ_{s}+Lie(X)J_{s}+V\phi(J_{s})+Lie(X)\phi(J_{s})\\
           &+J_{s}V+\phi(J_{s})V+\phi(J_{s})V+\phi(J_{s})Lie(X).
\end{align*}
We note that 
\[ V\phi(J_{s})=V \dashv J_{s}, ~ Lie(X)J_{s}=V \dashv J_{s} ,\]
\[ \phi(J_{s})V=J_{s} \vdash V, ~ J_{s}Lie(X)=J_{s} \vdash V ,\]
and $VJ_{s}+J_{s}V=0$. Therefore
\begin{align*}
    I_{s+1}&=J_{s+1}+\phi(J_{s})+\phi(J_{s})Lie(X)+Lie(X)\phi(J_{s})\\
           &=J_{s+1}+\phi(J_{s+1}).
\end{align*}
We have shown that $I_{s}=J_{s}+\phi(J_{s})$ for all $s \geq 0$, then $I_{s}\cap V=J_{s}$ which implies $I \cap V=J$.
\end{proof}

\begin{corollary}\label{corollary1}
di-$Lie\langle X \mid S \rangle$ is isomorphic to the subalgebra of $({Lie \langle X \cup \dot{X} \rangle} / I)^{(2)}$ generated by $\dot{X}$.
\end{corollary}
\section{Replication of HNN-Extension}
The initial approach in~\cite{L2} for constructing HNN-extensions of Leibniz algebras is based on the construction of HNN-extensions in the case of di-algebras, which are closely related to Leibniz algebras, just like associative algebras are related to Lie algebras. The results were transferred to Leibniz algebras by the PBW theorem. In this section, we introduce HNN-extensions of Leibniz algebras through  a replication procedure based on operads. The benefit in doing so is that we find an  explicit linear basis for HNN-extensions of Leibniz algebras.

Let $L$ be a right Leibniz algebra over a field $\mathbb{K}$ with some bilinear product $[-,-]$ satisfying the Leibniz identity
\[ [[x,y],z]=[[x,z],y]+[x,[y,z]].\]
The HNN-extension of Leibniz algebras is defined as \[L_{d}^{*}=\langle L,t \mid  [a,t]=d(a), ~\text{for} ~\text{all} ~a\in A\rangle ,\] where $d$ is a derivation map defined on the subalgebra $A$ instead of the whole algebra $L$. Also, Ladra \emph{et.~al}\cite{L2} defined HNN-extensions of Leibniz algebras corresponding to an anti-derivation map in order to prove that every Leibniz algebra of at most countable dimension is embeddable in a two-generator Leibniz algebra. 

We denote the generating sets of $L$ and $A$ by $X$ and $B$, respectively. Arbitrary elements of $X$ will be denoted by $x$, $y$, $z$, elements of $B$
by $a$, $b$. We recall that the replication of the skew-symmetric and the Jacobi identities yields the following identities in the variety of di-$\mathcal{L}{i}{e}$:
\begin{equation}
    [x \vdash y]+[y \dashv x] =0
\end{equation} 
and 
\begin{equation}\label{relations1}
    [[x \dashv y]\dashv z]-[x \dashv [y \dashv z]]-[[x \dashv z]\dashv y]=0.
\end{equation}
The latter equation is the right Leibniz identity. Therefore, a di-Lie algebra in the variety di-$\mathcal{L}{i}{e}$ is considered as a left Leibniz algebra with respect to $\vdash$ or a right Leibniz algebra with respect to $\dashv$. Derivations and anti-derivations of di-Lie algebras are defined as linear maps $d:L \to L$ and $d^{\prime}: L \to L$ satisfying 
\[ d([x\dashv y])=[d(x)\dashv y]+[x\dashv d(y)],\]
\[ d^{\prime}([x\dashv y])=[d^{\prime}(x)\dashv y]-[d^{\prime}(y) \dashv x],\]
\text{for all}~ $x,y \in L$. Therefore right multiplication is a derivation, whereas left multiplication is an anti-derivation.

We consider the following presentation for HNN-extensions of a Lie di-algebra in the variety di-$\mathcal{L}{i}{e}$ with operations $[a \dashv b] = [ab]$ and $[b \vdash a] = -[ba]$ and denote it by $H$. We have
\begin{align}\label{hnn}
    H = \langle X,t  \mid & [x\dashv y], [y \dashv x], 
     [x \dashv x], ~\text{for all} ~x,y \in X \\
     &[t\dashv a]-d^{\prime}(a), [a \dashv t]-d(a) ~ \text{and} ~a \in A\rangle, \nonumber
\end{align}
where $d$ and $d^{\prime}$ are derivation and anti-derivation maps, respectively, defined on the subalgebra $A$.
Let $H_{0}=span \{[x \vdash y] -[x \dashv y] \mid x,y \in L\}$ with basis $X_{0}$. We consider $X \cup \dot{X} \cup \{t, \dot{t}\}$, where $\dot{X}$ is a copy of $X$ and define the ordering $\dot{x}> \dot{y} > x>y$ and $\dot{t} > t$ such that $\{t, \dot{t}\} > X \cup \dot{X}$. Then we consider $S$ as the set of the following polynomials:
\begin{enumerate}
    \item $ \dot{f_{1}}(x,y)=[\dot{x} y]- {\dot{\mu}}_{\dashv}(x,y)$
    \item $\dot{f_{2}}(y,x)=[\dot{y}x]-{\dot{\mu}}_{\dashv}(y,x)$
    \item $\dot{f_{3}}(x,x)=[\dot{x} x]-{\dot{\mu}}_{\dashv}(x,x)$
    \item $\dot{g}(a,t)=[\dot{a} t]-{\dot{\mu}}_{\dashv}(a,t)$
    \item $\dot{h}(t,a)=[\dot{t} a]-{\dot{\mu}}_{\dashv}(t,a)$
\end{enumerate}
where by $\mu: X \times X \to X$ we denote the multiplication table of a Lie algebra which is a linear form in $X$ for all $x,y \in X$. We have $\phi(\dot{f_{1}})=[xy]-\mu_{\dashv}(x,y) $, $\phi(\dot{f_{2}})=[xy]+\mu_{\dashv}(y,x)$, 
$\phi(\dot{f_{3}})=\mu_{\dashv}(x,x)$, 
$\phi(\dot{h})=[ta]-\mu_{\dashv}(t,a)$ and $\phi(\dot{g})=[ta]+\mu_{\dashv}(a,t)$
such that $x,y,t,a \in X \setminus X_{0}$
and consider
\begin{equation*}
    S \cup \phi(S)=\{\dot{f_{1}},\dot{f_{2}}, \dot{f_{3}}, \dot{h}, \dot{g}, \phi(\dot{f_{1}}),\phi(\dot{f_{2}}), \phi(\dot{f_{3}}), \phi(\dot{h}), \phi(\dot{g}) \}.
\end{equation*}
The elements of $\phi(S)$ have the inclusion compositions ${\mu}_{\dashv}(x,y) + {\mu}_{\dashv}(y,x)$, ${\mu}_{\dashv}(x,x)$ and ${\mu}_{\dashv}(t,a)+ {\mu}_{\dashv}(a,t)$. Since the linear space spanned by these compositions coincides with $H_{0}$,  we add letters $X_{0}$ to $S \cup \phi(S)$. Moreover, by $\mu_{\dashv}(X, X_{0} ) = 0$ and $\mu_{\dashv}(X_{0},X) \subset X_{0}$ we can reduce $S \cup \phi(S)$ to the following set
\begin{equation}\label{gsb1}
     S \cup \phi(S)=\{\dot{f_{1}},\dot{f_{2}}, \dot{f_{3}}, \dot{h}, \dot{g}, \phi(\dot{f_{1}}), \phi(\dot{f_{3}}), \phi(\dot{h})\}.
\end{equation}
In the next theorem, we implement Kolesnikov's approach in order to compute a Gr\"obner-Shirshov basis of $H$.
\begin{theorem}
The relations in $S \cup \phi(S) \cup X_{0}$ in \ref{gsb1} form a Gr\"obner-Shirshov basis for $H$.
\end{theorem}
\begin{proof}
The relations $[xy]-\mu_{\dashv}(x,y)$ for $x,y \in X \setminus X_{0}$, $x>y$ and $[ta]-\mu_{\dashv}(t,a)$ for $t,a \in X \setminus X_{0}$, $t>a$ correspond to the multiplication table of the Lie algebra $\bar{H}=H/H_{0}$ and their intersection compositions are trivial. Considering \ref{relations1} and the relations between multiplications of a Lie algebra, we compute other possible compositions as follows:
\begin{itemize}
    \item [(i)] $\dot{f_1}=\dot{x}y-y\dot{x}-\dot{\mu}_{\dashv}(x,y)$, $\phi(\dot{f_1})=yz-zy-\mu_{\dashv}(y,z)$, $w=\dot{x}yz$, where $y,z \in X\setminus X_{0}$, $y>z$ and $x \in X$
    \begin{align*}
        (\dot{f_1},\phi(\dot{f_1}))_{w}&=\dot{f_1} z-\dot{x}\phi(\dot{f_1})=-y\dot{x}z-{\dot{\mu}_{\dashv}}(x,y)z + \dot{x}zy + \dot{x}\mu_{\dashv} (y,z)\\
        &\equiv -yz\dot{x}-y{\dot{\mu}}_{\dashv}(x,z)-{\dot{\mu}}_{\dashv}(x,y)z + z\dot{x}y+ {\dot{\mu}_{\dashv}}(x,z)y+ \dot{x}{\dot{\mu}}_{\dashv}(y,z)\\
        & \equiv -zy\dot{x} -\mu_{\dashv}(y,z)\dot{x}-y{\dot{\mu}}_{\dashv}(x,z) - {\dot{\mu}}_{\dashv}(x,y)z +zy\dot{x}+z {\dot{\mu}}_{\dashv}(x,y)\\
        &+{\dot{\mu}}_{\dashv}(x,z)y+\dot{x}{\mu}_{\dashv}(y,z)\\
        &={\dot{\mu}}_{\dashv}({\dot{\mu}}_{\dashv}(x,z),y)+{\dot{\mu}}_{\dashv}(\dot{x},{{\mu}}_{\dashv}(y,z))+{\dot{\mu}}_{\dashv}(z,{\dot{\mu}}_{\dashv}(x,y))=0
    \end{align*}
    \item [(ii)] $\dot{f_1}=\dot{x}y-y\dot{x}-\dot{\mu}_{\dashv}(x,y)$, $\phi(\dot{f_2})=yz-zy+\mu_{\dashv}(z,y)$, $w=\dot{x}yz$, where $y,z \in X\setminus X_{0}$, $y>z$ and $x \in X$
    \begin{align*}
        (\dot{f_1},\phi(\dot{f_2}))_{w}&=\dot{f_1} z-\dot{x}\phi(\dot{f_2})=-y\dot{x}z-{\dot{\mu}_{\dashv}}(x,y)z + \dot{x}zy - \dot{x}\mu_{\dashv} (z,y)\\
        &\equiv -\dot{x}yz+{\dot{\mu}}_{\dashv}(x,y)z - {\dot{\mu}_{\dashv}}(x,y)z+ \dot{x}zy -\dot{x}{{\mu}}_{\dashv}(z,y)=0
    \end{align*}
    \item [(iii)] $\dot{f_3}=\dot{x}x-x\dot{x}-\dot{\mu}_{\dashv}(x,y)$, $\phi(\dot{f_1})=xy-yx-\mu_{\dashv}(x,y)$, $w=\dot{x}xy$, where $x,y \in X\setminus X_{0}$, $x>y$
    \begin{align*}
        (\dot{f_3},\phi(\dot{f_1}))_{w}&=\dot{f_3} y-\dot{x}\phi(\dot{f_1})=-x\dot{x}y-{\dot{\mu}_{\dashv}}(x,x)y + \dot{x}yx + \dot{x}\mu_{\dashv} (x,y)\\
        &\equiv -xy\dot{x}-x{\dot{\mu}}_{\dashv}(x,y)-{\dot{\mu}}_{\dashv}(x,x)y + y\dot{x}x+ {\dot{\mu}_{\dashv}}(x,y)x+ \dot{x}{\dot{\mu}}_{\dashv}(x,y)\\
        & \equiv (yx-xy)\dot{x} -x{\dot{\mu}}_{\dashv}(x,y)-{\dot{\mu}}_{\dashv}(x,x)y + y{\dot{\mu}_{\dashv}}(x,x)+ {\dot{\mu}_{\dashv}}(x,y)x+ \dot{x}{\dot{\mu}}_{\dashv}(x,y)\\
        &={\dot{\mu}}_{\dashv}(y,{\mu}_{\dashv}(x,x))+{\dot{\mu}}_{\dashv}(\dot{x},{{\mu}}_{\dashv}(x,y))-{\dot{\mu}}_{\dashv}({\mu}_{\dashv}(x,y),\dot{x})=0
    \end{align*}
    
    \item [(iv)] $\dot{h}=\dot{t}a-a\dot{t}-\dot{\mu}_{\dashv}(t,a)$, $\phi(\dot{f_1})=ab-ba-\mu_{\dashv}(a,b)$, $w=\dot{t}ab$, where $a,b \in X\setminus X_{0}$, $a>b$
    \begin{align*}
        (\dot{h},\phi(\dot{f_1}))_{w}&=\dot{h}b-\dot{t}\phi(\dot{f_{1}})=-a\dot{t}b-{\dot{\mu}_{\dashv}}(t,a)b + \dot{t}ba + \dot{t}\mu_{\dashv} (a,b)\\
        &\equiv -ab\dot{t} -  a\dot{\mu}_{\dashv} (t,b) - \dot{\mu}_{\dashv} (t,a)b + b\dot{t}a + \dot{\mu}_{\dashv} (t,b)a + \dot{t} {\mu}_{\dashv}(a,b) \\
        & \equiv -ab\dot{t} -  a\dot{\mu}_{\dashv} (t,b) - \dot{\mu}_{\dashv} (t,a)b + ba\dot{t} + b {\dot{\mu}}_{\dashv} (t,a)+ \dot{\mu}_{\dashv} (t,b)a + \dot{t} {\mu}_{\dashv}(a,b) \\
        &=\dot{\mu}_{\dashv}(\dot{\mu}_{\dashv}(t,b),a)+\dot{\mu}_{\dashv}(b,\dot{\mu}_{\dashv}(t,a))-\dot{\mu}_{\dashv}(t,{\mu}_{\dashv}(a,b))=0
    \end{align*}
    
    \item [(v)] $\dot{g}=\dot{a}t-t\dot{a}-\dot{\mu}_{\dashv}(a,t)$, $\phi(\dot{h})=ta-at-\mu_{\dashv}(t,a)$, $w=\dot{a}ta$, where $a \in X\setminus X_{0}$, $t>a$
    \begin{align*}
        (\dot{g},\phi(\dot{h}))_{w}&=\dot{g}a-\dot{a}\phi(\dot{h})=-t\dot{a}a-{\dot{\mu}_{\dashv}}(a,t)a + \dot{a}at + \dot{a}\mu_{\dashv} (t,a)\\
        &\equiv -ta\dot{a} -t\dot{\mu}_{\dashv} (a,a) - \dot{\mu}_{\dashv} (a,t)a + a\dot{a}t + \dot{\mu}_{\dashv} (a,a)t + \dot{a} {\mu}_{\dashv}(t,a) \\
        & \equiv -at\dot{a} -{\mu}_{\dashv}(t,a)\dot{a} + t\dot{\mu}_{\dashv}(a,a) - \dot{\mu}_{\dashv} (a,t)a+ a\dot{a}t + \dot{\mu}_{\dashv} (a,a)t + \dot{a}\dot{\mu}_{\dashv}(t,a)\\
        &\equiv  -{\mu}_{\dashv}(t,a)\dot{a} + t\dot{\mu}_{\dashv}(a,a) - \dot{\mu}_{\dashv} (a,t)a + a \dot{\mu}_{\dashv}(a,t) + \dot{\mu}_{\dashv} (a,a)t + \dot{a}\dot{\mu}_{\dashv}(t,a)\\
        &=\dot{\mu}_{\dashv}(\dot{\mu}_{\dashv}(a,a),t)+\dot{\mu}_{\dashv}(a,{\mu}_{\dashv}(t,a)+\dot{\mu}_{\dashv}(a,\dot{\mu}_{\dashv}(a,t))=0
    \end{align*}

\end{itemize}

    There is no composition between elements of $\{ \dot{f_{1}},\dot{f_{2}},\dot{f_{3}},\dot{h}, \dot{g} \}$. We denote, for example, $(a \wedge b)$ the composition of the polynomials of type $(a)$
and type $(b)$. The intersection compositions $\dot{f_{2}} \wedge \phi(\dot{f_1})$, $\dot{f_{2}} \wedge \phi(\dot{f_2})$, $\dot{f_{3}} \wedge \phi(\dot{f_2})$, $\dot{h} \wedge \phi(\dot{f_2})$ and $\dot{g} \wedge \phi(\dot{g})$ are trivial modulo $S \cup \phi(S) \cup X_{0}$ similar to the cases $(i)$, $(ii)$, $(iii)$, $(iv)$ and $(v)$, respectively. Also, by  straightforward computation we observe that the inclusion compositions $\dot{f_{1}} \wedge \phi(\dot{f_3})$, $\dot{f_{2}} \wedge \phi(\dot{f_3})$, $\dot{f_{3}} \wedge \phi(\dot{f_3})$ and $\dot{h} \wedge \phi(\dot{f_3})$ are trivial modulo $S \cup \phi(S) \cup X_{0}$.
\end{proof}
From compositions of the presentation (\ref{hnn}) and from the CD-Lemma \ref{cdlemma}, we get a normal form for the elements of the HNN-extension.
\begin{corollary}
A linear basis for $H$ is given by all the Lyndon-Shirshov words on $X \cup \dot{X} \cup \{t,\dot{t} \}$ which do not contain subwords from the set $X_0$ or of the form $xy$ for $x,y \in X \setminus X_0$ and $x>y$, $ta$ for $a \in B$, $\dot{x}y$ and $\dot{x}x$ for $x,y \in X$ and $\dot{t}a$ for $a\in B$.
\end{corollary}

\begin{corollary}
The isomorphic copy of the Lie di-algebra $\dot{L}$ is embedded in $H$.
\end{corollary}
\begin{proof}
All the elements of $\dot{X}$ are words in normal form.
\end{proof}


\begin{thebibliography}{9}

\bibitem{B1} D.~W.~Barnes, \emph{On Levi's theorem for Leibniz algebras}, {Bull.~Austral.~Math.~Soc.}, 86(2), 184--185 (2012)

\bibitem{B2} A.~Bloh, \emph{A generalization of the concept of a Lie algebra}, Sov.~Math.~Dokl.~6, 1450--1452 (1965)

\bibitem{B3} L.~A.~Bokut, \emph{Imbeddings into simple associative algebras} (in Russian), Algebra i Logika 15(2), 117--142 (1976) 

\bibitem{B4} L.~Bokut, Y.~Chen, K.~Kalorkoti, P.~S.~Kolesnikov, V.~E.~Lopatkin, 
\emph{Gr\"obner-Shirshov bases; Normal Forms, Combinatorial and Decision Problems in Algebra}, World Scientific Publishing Company, (2020)

\bibitem{B5} L.~A.~Bokut, \emph{Gr\"obner-Shirshov bases for Lie algebras: After A.~I.~Shirshov}, Southeast Asian Bulletin of Mathematics 31, 1057--1076 (2007)




Contemp.~Math., vol.~623, Amer.~Math.~Soc., Providence, RI, 41--54 (2014)


\bibitem{G2} V.~Yu.~Gubarev, P.~S.~Kolesnikov, \emph{Operads of decorated trees and their duals}, Commentat.~Math.~Univ.~Carolin.~55(4), 421--445 (2014)

\bibitem{G3} V.~Yu.~Gubarev, P.~S.~Kolesnikov. \emph{On the computation of Manin products for operads},  arXiv:1204.0894v1 [math.RA]

\bibitem{H1} G.~Higman, B.~H.~Neumann, H.~Neumann, \emph{Embedding theorems for groups}, J.~London.~Math.~Soc.~ 24, 247--254 (1949) 

\bibitem{K1} P.~S.~Kolesnikov, \emph{Gr\"obner-Shirshov bases for replicated algebras}, Algebra Colloquium 24(4), 563--576 (2017) 


\bibitem{L1} M.~Ladra, P.~P\'aez-Guill\'an, C.~Zargeh, \emph{HNN-extension of Lie superalgebras}, Bull.~Malays.~Math.~Sci.~Soc. 43, 959--1970 (2020)

\bibitem{L2} M.~Ladra, M.~Shahryari, C.~Zargeh, \emph{HNN-extension of Leibniz algebras}, J.~Algebra 532, 183--200 (2019)

\bibitem{L3} A.~I.~Lichtman, M.~Shirvani, \emph{HNN-extensions for Lie algebras}, Proc.~AMS.~125(12), 3501--3508 (1997)

\bibitem{L4} J.-L, Loday, \emph{Une Version non commutative des algebras de Lie: les algebras de Leibniz}, Enseign.~Math.~39, 269--293 (1993)

\bibitem{L5} J.-L. Loday, \emph{Cyclic homology}, Grundl.~Math.~Wiss.~301, Springer-Verlag, Berlin, (1992)


\bibitem{L7} J.-L. Loday, B.~Vallette, \emph{Algebraic Operads}, Grundlehren Math.~Wiss.~346, Springer, Heidelberg (2012)

\bibitem{S1} A.~I. Shirshov. \emph{Some algorithmic problems for Lie algebras}, Sibirsk.~Mat.~Zh.~3, (1962), 292--296 (in Russian). English translation: SIGSAM Bull.~33, 3--6 (1999)

\bibitem{S2} A.~I.~Shirshov, \emph{Selected works of A.~I.~Shirshov. Translated by Murray Bremner and Mikhail V.~Kotchetov},  L.~A.~Bokut, V.~Latyshev, I.~Shestakov, E.~Zelmanov (Eds), {Contemporary Mathematicians}, Birkhaeuser Verlag, Basel (2009)

\bibitem{S3} S.~Silvestrov, C.~Zargeh, \emph{HNN-extension of involutive multiplicative Hom-Lie algebras}, arXiv:2101.01319 [math.RA]

\bibitem{W1} A.~Wasserman, \emph{A derivation HNN construction for Lie algebras}, Israel J.~Math.~106, 76--92 (1998)

\bibitem{Z1} C.~Zargeh, \emph{Existentially closed Leibniz algebras and an embedding theorem}, Commun.~Math.~29(2), 163--170 (2021)

\bibitem{Z2} G.W. Zinbiel, Encyclopedia of types of algebras 2010, Nankai Ser. Pure Appl. Math. Theoret. Phys., 9 (2012), 217 − 297.


\end{thebibliography}
\end{document}